\documentclass[10pt,twoside]{article}

\usepackage{amssymb}
\usepackage{amsmath, amsthm}
\usepackage{amsfonts}
\usepackage[utf8]{inputenc}
\usepackage[normalem]{ulem}
\usepackage{graphicx}
\usepackage{latexsym}
\usepackage{pdfpages}
\usepackage{makeidx}
\usepackage{float}
\usepackage{tikz}
\usepackage{mathtools}
\newtheorem{theorem}{Theorem}[section]
\newtheorem{definition}{Definition}[section]
\newtheorem{lemma}{Lemma}[section]
\newtheorem{proposition}{Proposition}[section]
\newtheorem{remark}{Remark}[section]
\newtheorem{corollary}{Corollary}[section]
\makeatletter\def\theequation{\arabic{section}.\arabic{equation}}\makeatother

\begin{document}
\title{
{\bf\Large  On Bifurcation of Solutions of the Yamabe Problem in Product Manifolds with Minimal Boundary}}
\author{{\bf\large Elkin Dario C\'ardenas Diaz}\footnote{The author is partially sponsored by CAPES-Brazil via Instituto de Matem\'atica e Estat\'istica da Universidade de S\~ao Paulo.} \vspace{1mm}\\
{\it\small Departamento de Matem\'aticas}\\ {\it\small Universidad del Cauca},
{\it\small Popay\'an, Colombia, 190003}\\
{\it\small e-mail: ecardenas@unicauca.edu.co, elkindarioc@gmail.com}\vspace{1mm} \\
{\bf\large Ana Cl\'audia da Silva Moreira}\vspace{1mm}\\
{\it\small Departamento de Matem\'atica}\\ {\it\small Universidade de S\~ao Paulo},
{\it\small S\~ao Paulo, Brazil, 05508-090}\\
{\it\small e-mail: amoreira@ime.usp.br, anaclaudia77br@gmail.com}}\vspace{1mm}
\maketitle
\begin{center}
{\bf\small Abstract}
\vspace{3mm}
\hspace{.05in}\parbox{4.5in}
{{\small In this paper, we study multiplicity of solutions of the Yamabe problem on product manifolds with minimal boundary via bifurcation theory.
}}
\end{center}

\section{\bf Introduction}
\def\theequation{1.\arabic{equation}}\makeatother
\setcounter{equation}{0}

Geometers have been interested in finding canonical metrics on a Riemannian manifold for years. Probably, the best known problem on this topic is the Uniformization Problem, for compact surfaces, that assures that there exists a metric of constant Gaussian Curvature in each conformal class. For higher dimensions, a Japanese mathematician named Hidehiko Yamabe proposed a question, later known as Yamabe's Problem:

\textit{Let $(M,g_{0})$ be a compact Riemannian manifold (without boundary) of dimension $m \geq 3$. Is there a metric $g$, conformal to $g_{0}$, having constant scalar curvature?}

In 1960, it was published an article \cite{YA60}, authored by Yamabe, with a proof of the statement. But in 1968, Trudinger found an error in Yamabe's proof, and gave a complete argument for the existence of a solution to the problem for non-positive scalar curvature \cite{TR68}. Only in 1989, the combined work of Aubin \cite{AU76}, Trudinger \cite{TR68}, and Schoen \cite{SC84}, \cite{SC89} lead a complete proof of the existence of a solution for the Yamabe problem in its full generality. It is known that the critical points of the Hilbert-Einstein functional $F$, restricted to the set of metrics  on $M$ having unit volume, are the Einstein metrics of volume 1 on $M$; and the critical points of $F$, restricted to the metrics conformal to $g$ with volume 1, are the constant scalar curvature metrics of volume 1.

From then on, many questions have been raised about uniqueness and multiplicity of solutions, compactness and non-compactness of the set of solutions and so on. Moreover, variations of the problem have been proposed by several authors. Nowadays, a great amount of results on the Yamabe problem and its generalizations can be found in the literature. Our starting problem originates from an article of Lima, Piccione and Zedda \cite{PI12}, where the authors studied local rigidity and multiplicity of constant scalar curvature metrics in compact product manifolds, using a bifurcation result. Given a family $\{g_{s}\}_{s \in [a,b]}$ of solutions for the Yamabe Problem, on the product manifold, the existence of a bifurcation instant $s_{*}$ gives an entirely new sequence $({g}_{n})_{n > 0}$ of solutions! The authors proved that if $g_{s_{*}}$ is a degenerate critical point of $F$, restricted to the conformal class (with unitary volume), $s_{*}$ is a candidate to a bifurcation instant. In fact, except for a finite number of instants $s$ for which $g_{s}$ is degenerate, all the others are bifurcation instants.

Bifurcation techniques in the Yamabe problem have been used in the case of collapsing Riemannian manifolds, and in the case of homogeneous metrics on spheres, see references \cite{BE13} e \cite{PI13}. There are several possible formulations of the  Yamabe problem in manifolds with boundary $(M, \overline{g})$. Here we consider the following:
\textit{Let $(M, \overline{g})$ be a compact Riemannian manifold with boundary of dimension $m \geq 3$. The Yamabe problem in  $(M, \overline{g})$ consists in finding a metric $\tilde{g}$, conformal to $\overline{g}$, for which $M$ has constant scalar curvature and $\partial M$ has vanishing mean curvature.}
The question of existence and other aspects of solutions for the above problem have been studied in \cite{AL10}, \cite{DI14}, \cite{ES92a}, \cite{ES92b}, \cite{HA99}, among others. 
Writing a metric conformal to $\overline g$ as $\tilde{g} = u^{\frac{4}{n-2}} \overline{g}$, then $\tilde g$ is a solution with constant scalar curvature and vanishing mean curvature of the Yamabe problem on manifolds with boundary, if and only if the function $u$ is a solution to the Neumann problem:
$$
\left\{
\begin{array}{c}
\displaystyle\frac{4(m-1)}{m-2} \Delta_{\overline{g}} u + R_{\overline{g}}u - K u ^{\frac{m+2}{m-2}} = 0,\ \ \mbox{em } M \\
(m-1) \partial_{\eta_{\overline{g}}} u + \displaystyle\frac{m-2}{2} H_{\overline{g}} u =  0 \ \ \mbox{em } \partial{M}
\end{array}
\right. $$
where $K$ is constant, $\Delta_{\overline{g}}$ is the Laplacian operator of $\overline{g}$, $R_{\overline{g}}$ is the scalar curvature of $\overline{g}$, $H_{\overline{g}}$ is the mean curvature of $\partial M$ relatively to $\overline{g}$, and $\eta_{\overline{g}}$ is the $\overline{g}$-unit normal field along $\partial M$ pointing inside of $M$. Solving the above problem is equivalent to finding critical points for the Hilbert-Einstein functional, $F: \mathcal{M}^{k, \alpha}(M) \longrightarrow \mathbb{R},$ defined by
$$F(g) = \displaystyle\int_M R_g \omega_g,$$
restricted to the set of metrics in the conformal class of $\overline{g}$ having unit volume.

The aim of  this paper is to determine multiplicity of solutions of the Yamabe problem in manifolds obtained as a product of a compact manifold (without boundary) and a compact manifold with boundary, via bifurcation theory. 
We consider a setup similar to the problem studied in reference \cite{PI12}. We consider here the case of manifolds with boundary, which introduces new elements with respect to the theory developed in \cite{PI12}: given $(M_{1}, g^{(1)})$, a compact Riemannian manifold with $\partial M_{1} = \emptyset$ and $(M_{2}, \overline{g}^{(2)})$ a compact Riemannian manifold with minimal boundary, both having constant scalar curvature, consider the product manifold, $M = M_{1} \times M_{2}$, whose boundary is given by $\partial M = M_{1} \times \partial M_{2}$. Let $m_{1}$ and $m_{2}$ be the dimensions of $M_{1}$ and $M_{2}$, respectively, and assume that $\mbox{dim}(M) = m = m_{1} + m_{2} \geq 3$. For each $s \in (0, + \infty)$, define ${\overline{g}}_{s} = {g}^{(1)} \oplus s {\overline{g}}^{(2)}$ a family of metrics on $M$. Then $(M, \overline{g}_{s})$ has constant scalar curvature, and the mean curvature of $\partial M$ relatively to $\overline g_s$ vanishes for all $s>0$.

The main difficulty faced in the transition from the case without boundary to the case with boundary was to find a result of uniqueness (to ensure rigidity) that can be well adapted to the case with boundary. We know several results about uniqueness. Escobar \cite{ES90}, for instance, proved that if there is a metric $\tilde{g}$ with constant scalar curvature and vanishing mean curvature in the $\mathcal{C}^{k, \alpha}$-conformal class of an Einstein metric $\overline{g}$, then $\tilde{g}$ is Einstein. Moreover, if $\overline{g}$ is not conformal to the round metric, $\tilde{g}$ is unique, except for homothety. We find more appropriate to adapt the result, proved by Lima, Piccione and Zedda \cite{LI12}, which states that every $\mathcal{C}^{k, \alpha}$-conformal class of metrics of volume one has at most one metric of volume one and constant scalar curvature in a neighborhood of a non-degenerate metric with constant scalar curvature.

The main results in this paper, Theorem~\ref{thm:principal} and Theorem~\ref{thm:secondmain},  state that, when the scalar curvature of each factor $(M_i,g^{(i)})$ is positive, $i=1,2$, then  there are two sequences (one tending to zero and the other one to $+\infty$) of instants $s_{*} \in (0, +\infty)$ such that the corresponding metric $\overline g_{s_*}$ is the limit of a sequence of metrics of distinct solutions of the Yamabe problem. Precise definitions of these bifurcating branches of solutions will be given below.
In particular, we have multiplicity of solutions to the Yamabe problem on the normalized conformal classes of these new metrics. For all other values of $s$, the family is locally rigid, which means that, locally, the metrics of the family are the unique solutions to the Yamabe problem, up to homotheties.

The paper is organized as follows: in Section 2, we study the variational framework used for the bifurcation result. This framework is given by considering the Hilbert-Einstein functional restricted to the normalized conformal classes.  In Section 3 we discuss some results about rigidity and bifurcation that are used to obtain the conclusions of this work. Finally, in Section 4, we can verify that, also in the case of manifolds with boundary, essentially the same results obtained by Lima, Piccione and Zedda \cite{PI12} remain valid.

The authors gratefully acknowledge the supervision of Prof. Paolo Piccione throughout the development of this work.

\section{\bf General Settings}
\def\theequation{2.\arabic{equation}}\makeatother
\setcounter{equation}{0}

\subsection{Manifolds and Conformal Classes}

Let $(M, \overline{g})$ be a $m$-dimensional oriented compact Riemannian manifold with boundary $\partial M \neq \emptyset$, $m \geq 3$. As the metric $\overline{g}$ induces inner products and norms in all spaces of tensors on $M$ and Levi-Civita connection $\overline{\nabla}$ of $\overline{g}$ induces a connection in all vector spaces of tensor fields on $M$, the space $\Gamma^{k} (T^*M \otimes T^*M)$ of $\mathcal{C}^{k}$-sections of the vector bundle $T^*M \otimes T^*M$ of symmetric $(0,2)$-tensors of class $\mathcal{C}^k$ on $M$ is a Banach space with the norm
$$\left\| \tau \right\|_{\mathcal{C}^{k}} = \max_{j=0, \ldots, k} \left[ \max_{p \in M} \left\| \overline{\nabla}^{(j)} \tau(p) \right\|_{\overline{g}} \right],$$
and therefore, it is a Banach manifold.

Given $k \geq 3$ and $\alpha \in (0,1]$, denote by $\mathcal{M}^{k, \alpha}(M)$ the set of all Riemannian metrics of class $\mathcal{C}^{k, \alpha}$ on $M$, in the sense that the coefficients of the metrics are $\mathcal{C}^{k, \alpha}$-functions on $M$.
The set $\mathcal{M}^{k, \alpha}(M)$ is an open cone of $\Gamma^{k, \alpha} (T^*M \otimes T^*M)$, so this is a Banach manifold itself and $T_{\overline{g}}\mathcal{M}^{k, \alpha}(M) = \Gamma^{k, \alpha} (T^*M \otimes T^*M)$, for a metric $\overline{g} \in \mathcal{M}^{k, \alpha}(M)$.

Consider the open subset of the Banach space $\mathcal{C}^{k, \alpha}(M)$,
$$\mathcal{C}_{+}^{k, \alpha}(M) = \left\{ \phi \in \mathcal{C}^{k, \alpha}(M) : \phi > 0 \right\}.$$
Now, for each $\overline{g} \in \mathcal{M}^{k, \alpha}(M)$, define the $\mathcal{C}^{k, \alpha}$-conformal class of $\overline{g}$ by
$$[\overline{g}] \colon = \{ \phi \overline{g} : \phi \in \mathcal{C}^{k, \alpha}_{+}(M) \}.$$

\begin{proposition}
	The $\mathcal{C}^{k, \alpha}$- conformal class of a Riemannian metric $\overline{g} \in \mathcal{M}^{k, \alpha}(M)$ is an embedded submanifold of $\mathcal{M}^{k, \alpha}(M)$.
\end{proposition}
\noindent
\begin{proof}
\noindent
	Given $\overline{g} \in \mathcal{M}^{k, \alpha}(M)$. Consider the injective map
	$$
	\begin{array}{rcl}
	\mathcal{I}_{\overline{g}} : \mathcal{C}^{k, \alpha}_{+}(M) & \longrightarrow & \mathcal{M}^{k, \alpha}(M) \\
	\phi & \longmapsto & \phi \overline{g}
	\end{array}
	$$
	whose differential $\left( \mbox{d}\mathcal{I}_{\overline{g}} \right)_{\phi} : \mathcal{C}^{k, \alpha}(M)  \longrightarrow  \Gamma^{k, \alpha}(T^*M \otimes T^*M)$
	is injector and has a left inverse given by the linear bounded operator
	$$
	\begin{array}{rcl}
	\mathcal{J}_{\overline{g}} : \Gamma^{k, \alpha}(T^*M \otimes T^*M) & \longrightarrow & \mathcal{C}^{k, \alpha}(M) \\
	h & \longmapsto & \displaystyle\frac{1}{m} \mbox{tr}_{\overline{g}} h;
	\end{array}
	$$
	consequently \footnote{Let $T:E\longrightarrow F$ be a linear bounded operator between Banach spaces. If $T$ admits a continuous right-inverse $S:F\longrightarrow E$, then, $\mbox{ker }T$ admits a closed complement; moreover, this complement is $\mbox{Im }S$.
	}, the image $\mbox{Im } d\mathcal{I}_{\overline{g}}$ has a closed complement in $\Gamma^{k, \alpha}(T^*M \otimes T^*M)$ and $\mbox{Im} \mathcal{I}_{\overline{g}} = [\overline{g}]$ is an embedded submanifold of $\mathcal{M}^{k, \alpha}(M)$.
\end{proof}

In particular, $[\overline{g}]$ is a Banach manifold with differential structure induced by $\mathcal{C}^{k, \alpha}(M)$ and its tangent space is
$$T_{\overline{g}}[\overline{g}] = \{\psi \overline{g} : \psi \in \mathcal{C}^{k, \alpha}(M) \},$$
which can be identified with $\mathcal{C}^{k, \alpha}(M)$.

For each $\overline{g} \in \mathcal{M}^{k, \alpha}(M)$, denote by $Ric_{\overline{g}}$ the Ricci curvature and $R_{\overline{g}}$ the scalar curvature with respect to $\overline{g}$. Let $\eta_{\overline{g}}$ be the unitary (inward) vector field normal to $\partial M$; denote by $H_{\overline{g}}$ the mean curvature of the boundary, induced by $\overline{g}$. They are $\mathcal{C}^{{k-2}, \alpha}$-functions, $Ric_{\overline{g}}$ and $R_{\overline{g}}$ defined on $M$ and $H_{\overline{g}}$ on $\partial M$. Let $\omega_{\overline{g}}$ be the volume form on $M$, with respect to $\overline{g}$, and $\sigma_{\overline{g}}$ the volume form induced on $\partial M$.

The volume function defined on $\mathcal{M}^{k, \alpha}(M)$ is
$$\mathcal{V}(\overline{g}) \colon = \displaystyle\int_{M} \ \omega_{\overline{g}}.$$
Observe that $\mathcal{V}$ is smooth and for each $\overline{g} \in \mathcal{M}^{k, \alpha}(M)$ and $h \in T_{\overline{g}}\mathcal{M}^{k, \alpha}(M)$ its differential is given by
\begin{equation}\label{eq:derVol}
(d\mathcal{V})_{\overline{g}}(h) = \displaystyle\frac{1}{2} \displaystyle\int_{M} \mbox{tr}_{\overline{g}}(h) \ \omega_{\overline{g}}.
\end{equation}

We define
$$\mathcal{M}^{k, \alpha}(M)_{1} \colon = \big\{ \overline{g} \in \mathcal{M}^{k, \alpha}(M) : \mathcal{V}(\overline{g}) = 1 \big\}$$
to be the subset of unitary volume metrics in $\mathcal{M}^{k, \alpha}(M)$.

\begin{proposition}\label{prop:emb}
	$\mathcal{M}^{k, \alpha}(M)_{1}$ is a smooth embedded submanifold of $\mathcal{M}^{k, \alpha}(M)$.
\end{proposition}
\noindent
\begin{proof}
\noindent
	Consider the smooth volume function $\mathcal{V}$ defined on $\mathcal{M}^{k, \alpha}(M)$. For $\overline{g} \in \mathcal{M}^{k, \alpha}(M)_{1}$, set $h = \overline{g}$, then, from (\ref{eq:derVol}) we get:
	$$(d\mathcal{V})_{\overline{g}}(\overline{g}) = \displaystyle\frac{1}{2} \displaystyle\int_{M} \mbox{tr}_{\overline{g}}(\overline{g}) \ \omega_{\overline{g}} = \displaystyle\frac{m}{2} \mathcal{V}(\overline{g}) = \displaystyle\frac{m}{2} \neq 0,$$
	the differential is surjective. So, $\mathcal{M}^{k, \alpha}(M)_{1}  = \mathcal{V}^{-1}(1)$ is the inverse image of a regular value. Besides, the kernel (the tangent space $T_{\overline{g}}\mathcal{M}^{k, \alpha}(M)_{1}$),
	$$\ker\{(d\mathcal{V})_{\overline{g}}\} = \left\{ h \in T_{\overline{g}}\mathcal{M}^{k, \alpha}(M) : \displaystyle\int_{M} \mbox{tr}_{\overline{g}}(h) \ \omega_{\overline{g}} = 0 \right\}$$
	has a closed complementary space; so the result follows.
\end{proof}

Observe that if $\overline{g} \in \mathcal{M}^{k, \alpha}(M)_{1}$, the conformal metric $\phi \overline{g}$, for some $\phi \in \mathcal{C}_{+}^{k, \alpha}(M)$, is not in $\mathcal{M}^{k, \alpha}(M)_{1}$, in general. Indeed,
$$\mathcal{V}(\phi \overline{g}) = \displaystyle\int_{M} \omega_{\phi \overline{g}} = \displaystyle\int_{M} \phi^{\frac{m}{2}} \ \omega_{\overline{g}}.$$
So, for each $\overline{g} \in \mathcal{M}^{k, \alpha}(M)_{1}$, we define
$$[\overline{g}]_{1} = \left\{ \phi \overline{g} : \phi \in \mathcal{C}_{+}^{k, \alpha}(M), \displaystyle\int_{M} \phi^{\frac{m}{2}} \ \omega_{\overline{g}} = 1 \right\},$$
which is an embedded submanifold of $[\overline{g}]$; the proof is similar to the Proposition~\ref{prop:emb}.

It is proved in \cite{LI12} that $[\overline{g}]_{1}$ is an embedded submanifold of $\mathcal{M}^{k, \alpha}(M)_{1}$.

\begin{proposition}
	The $\mathcal{C}^{k, \alpha}$-conformal class of metrics of unitary volume, $[\overline{g}]_{1}$, is an embedded submanifold of $\mathcal{M}^{k, \alpha}(M)_{1}$. Moreover,
	$$T_{\overline{g}}[\overline{g}]_{1} = \left\{ \psi \overline{g} : \psi \in \mathcal{C}^{k, \alpha}(M), \ \displaystyle\int_{M} \psi \ \omega_{\overline{g}} = 0 \right\}.$$
\end{proposition}
\noindent
\begin{proof}
\noindent
	As we know that $\mathcal{C}_{+}^{k, \alpha}(M)$ is a Banach manifold, for each $\overline{g} \in \mathcal{M}^{k, \alpha}(M)_{1}$, we can define $\mathcal{V}_{\overline{g}} : \mathcal{C}_{+}^{k, \alpha}(M) \longrightarrow \mathbb{R}$ by
	$$\mathcal{V}_{\overline{g}}(\phi) \colon = \mathcal{V}(\phi \overline{g}),$$
	a smooth function whose differential, for each $\psi \in T_{\phi}\mathcal{C}_{+}^{k, \alpha}(M) = \mathcal{C}^{k, \alpha}(M)$ is
	$$(d\mathcal{V}_{\overline{g}})_{\phi}(\psi) = \displaystyle\frac{m}{2} \displaystyle\int_{M} \phi^{\frac{m}{2}-1} \psi \ \omega_{\overline{g}}.$$
	
	Let $\phi \in \mathcal{C}_{+}^{k, \alpha}(M)$ be such that $\mathcal{V}_{\overline{g}}(\phi) = 1$. Take $\psi = \phi$ to see that $(d\mathcal{V}_{\overline{g}})_{\phi}(\phi) = \frac{m}{2} \neq 0$, i.e., $(d\mathcal{V}_{\overline{g}})_{\phi}$ is surjective; moreover, its kernel is complemented in $\mathcal{C}^{k, \alpha}(M)$, which implies that $\mathcal{V}_{\overline{g}}^{-1}(1)$ is an embedded submanifold of $\mathcal{C}_{+}^{k, \alpha}(M)$.
	
	Now, for $\overline{g} \in \mathcal(M)^{k, \alpha}(M)_{1}$ define the smooth maps $S: \mathcal{V}_{\overline{g}}^{-1}(1) \longrightarrow \mathcal{M}^{k, \alpha}(M)_{1}$ and $T: \mathcal{M}^{k, \alpha}(M)_{1} \longrightarrow \mathcal{C}^{k, \alpha}(M)$ by
	$$S(\phi) = \phi \overline{g} \ \ \  \mbox{ and } \ \ \ T(\tilde{g}) =  \frac{1}{m} \mbox{tr}_{\overline{g}} (\tilde{g}).$$
	Observe that $S$ is an injective immersion and $T$ is a smooth left-inverse for $S$. Besides that, $\mbox{Im }(dS)_{\phi}$ has a closed complement. Therefore, $[\overline{g}]_{1} = \mbox{Im }S$ is an embedded submanifold of $\mathcal{M}^{k, \alpha}(M)_{1}$.
	
	Finally, if we take $\phi = 1$, $S(1) = \overline{g}$ we have
	$$T_{1} \mathcal{V}_{\overline{g}}^{-1}(1) = \left\{ \psi \in \mathcal{C}^{k, \alpha}(M) : \displaystyle\int_{M} \psi \ \omega_{\overline{g}} = 0 \right\}.$$
	As $S$ is an immersion, $T_{\overline{g}}[\overline{g}]_{1} = \mbox{Im } (dS)_{1}$; but, $(dS)_{\phi}(\psi) = \psi \overline{g}$, for all $\phi \in \mathcal{V}_{\overline{g}}^{-1}(1)$, $\psi \in T_{\phi}\mathcal{V}_{\overline{g}}^{-1}(1)$, then we obtain the desired expression for the tangent space, which can also be identified with
	$$T_{\overline{g}}[\overline{g}]_{1} = \left\{ \psi \in \mathcal{C}^{k, \alpha}(M) : \ \displaystyle\int_{M} \psi \ \omega_{\overline{g}} = 0 \right\}.$$
\end{proof}

Note that $[\overline{g}]_{1} = \mathcal{M}^{k, \alpha}(M)_{1} \cap [\overline{g}]$. In \cite{LI12} it is also proved that $\mathcal{M}^{k, \alpha}(M)_{1}$ is transverse to $[\overline{g}]$, so $[\overline{g}]_{1}$ is an embedded submanifold of  $\mathcal{M}^{k, \alpha}(M)$.

Now, define the $\mathcal{C}^{k, \alpha}$-normalized conformal class of a metric $\overline{g} \in \mathcal{M}^{k, \alpha}(M)$ by
$$[\overline{g}]^{0} = \{\tilde{g} \in [\overline{g}] : H_{\tilde{g}} = 0 \}.$$
This is a non-empty set. Indeed, a result obtained by Escobar \cite{ES92a} assures that \textit{there is at least one metric with vanishing mean curvature in each conformal class}, so given a conformal class $[\overline{g}]$ we can assume that $H_{\overline{g}} = 0$.

\begin{proposition}
	The $\mathcal{C}^{k, \alpha}$-normalized conformal class of $\overline{g}$ can be identified with
	$$\mathcal{C}^{k, \alpha}_{+}(M)^{0} = \left\{\phi \in \mathcal{C}^{k, \alpha}_{+}(M) : \partial_{\eta_{\overline{g}}} \phi = 0 \mbox{ on } \partial M \right\},$$
	which is a closed subset of $\mathcal{C}^{k,\alpha}(M)$.
\end{proposition}
\noindent
\begin{proof}
\noindent
	Let $\tilde{g} = \phi^{\frac{4}{m-2}} \overline{g} \in [\overline{g}]^{0}$. We denote with tilde all quantities related with $\tilde{g}$. Then $\tilde{\eta} = \phi^{\frac{-2}{m-2}} \eta_{\overline{g}}$ and
	$$\tilde{H} = \tilde{g}^{ij} \tilde{g} (\tilde{\eta}, \tilde{\nabla}_{\partial_{i}} \partial_{j}) = \tilde{g}^{ij} \tilde{g}(\tilde{\eta}, \tilde{\Gamma}_{ij}^{r} \partial_{r}),$$
	but $\tilde{\Gamma}_{ij}^{r} = {\Gamma}_{ij}^{r} + \frac{2}{m-2} \phi^{-1} (\delta^{r}_{j} \partial_{i} \phi + \delta^{r}_{i} \partial_{j} \phi - \overline{g}_{ij} \nabla^{r} \phi)$ and $\tilde{g}^{ij} = \phi^{\frac{-4}{m-2}} \overline{g}^{ij}$. Then,
	$$\tilde{H} = \phi^{\frac{-2}{m-2}} \left( H_{\overline{g}} + \displaystyle\frac{2(m-1)}{(m-2)} \phi^{-1} \partial_{\eta_{\overline{g}}} \phi \right) = \phi^{-\frac{m}{m-2}} \left( \phi H_{\overline{g}} + \displaystyle\frac{2(m-1)}{(m-2)} \partial_{\eta_{\overline{g}}} \phi \right).$$
	As $H_{\overline{g}} = 0$ and $\phi > 0$, it follows that $\tilde{H} = 0$ if and only if $\partial_{\eta_{\overline{g}}} \phi = 0$.
\end{proof}

We want to show that the $\mathcal{C}^{k, \alpha}$-normalized conformal class, $[\overline{g}]^{0}$, is a submanifold of the $\mathcal{C}^{k, \alpha}$-conformal class, $[\overline{g}]$. To this aim we need the following Proposition which is an elementary version of a more general result that can be found in \cite[Theorem~4, Ch.\ IV]{ST70}.

\begin{proposition}
	There exists a continuous linear map
	$$\mathcal{F} : \mathcal{C}^{k,\alpha}(\partial M) \longrightarrow \mathcal{C}^{k+1,\alpha}(M)$$
	such that, for $\xi \in \mathcal{C}^{k,\alpha}(\partial M)$ the following properties are satisfied:
	\begin{enumerate}
		\item [(a.)] $\mathcal{F}(\xi)$ vanishes on $\partial M$;
		\item [(b.)] $\partial_{\eta} \mathcal{F}(\xi) = \xi$.
	\end{enumerate}
\end{proposition}
\begin{proof}
	Choose a finite set of local charts $(U_{r},\varphi_{r}),\ r = 1,\ldots,n$ on $M$ satisfying the following properties:
	\begin{enumerate}
		\item[(a.)] $U_{r}$ is an open subset of $M$, with $U_{r}\cap \partial M \neq \phi$ for all $r$;
		\item[(b.)] $U \colon = \bigcup\limits_{r=1}\limits^{n} U_{r}$ is an open neighborhood of $\partial M$;
		\item[(c.)] $\varphi_{r}$ is a diffeomorphism from $U_{r}$ to $\mathbb{R}^{m-1} \times [0,+\infty)$ carrying $U_{r}\cap \partial M$ onto $\mathbb{R}^{m-1} \times \{0\}$;
		\item[(d.)] $(d\varphi_{r})_{p}(\eta(p)) = \frac{\partial}{\partial x_m}$, for all $p \in U_{r} \cap\partial M$.
	\end{enumerate}
	Set $U_{0} = M \backslash \partial M$, so that $(U_{r})_{r=0,\ldots,n}$ is an open cover of $M$, and let $(\rho_{r})_{r=0}^{n}$ be a smooth partition of unity subordinated to such cover. Given $\xi \in \mathcal{C}^{k,\alpha}(\partial M)$, for all $r$, consider the function $\xi_{r}=\xi \circ \varphi_{r}^{-1} : \mathbb{R}^{m-1} \longrightarrow \mathbb{R}$, which is of class $\mathcal{C}^{k,\alpha}$. Let $F_{\xi_{r}} : \mathbb{R}^{m}\longrightarrow\mathbb{R}$ defined by
	$$F_{\xi_{r}}(x) = \frac{1}{x_{m}^{m-1}}\int_{Q(x)}\xi_{r}(z)dz,$$
	where, $x = (x_{1},\ldots,x_{m-1},x_{m})$, $Q(x)=\prod\limits_{i=1}\limits^{m-1}[x_{i}-\frac{1}{2}x_{m},x_{i}+\frac{1}{2}x_{m}]$, $z = (z_{1},\ldots,z_{m-1})$ $x_{m} \neq 0.$ And $F_{\xi_{r}}(x_{1},\ldots,x_{m-1},0)=0.$ A straightforward calculation shows that $F_{\xi_{r}} \in \mathcal{C}^{k+1,\alpha}(\mathbb{R}^{m})$. Let $\mathcal{F}_{r} = F_{\xi_{r}} \circ \varphi_{r}$; clearly $\mathcal{F}_{r} \in \mathcal{C}^{k+1,\alpha}(U_{r}).$ Finally, define $\mathcal{F}(\xi) : M \longrightarrow \mathbb{R}$ as
	$$\mathcal{F}(\xi) \colon = \sum\limits_{r=1}\limits^{n}\rho_{r}\cdot \mathcal{F}_{r}.$$
	It is easy to see that $\mathcal{F}(\xi)$ satisfies the desired properties.
\end{proof}

\begin{proposition}
	The $\mathcal{C}^{k, \alpha}$-normalized conformal class, $[\overline{g}]^{0}$, is an embedded submanifold of $[\overline{g}]$.
\end{proposition}
\begin{proof}
	Given $\overline{g} \in \mathcal{M}^{k, \alpha}(M)$, let $\eta_{\overline{g}}$ be the unitary (inward) vector field normal to the boundary $\partial M$. Define
	$$
	\begin{array}{rcl}
	\mathcal{N}_{\overline{g}} : [\overline{g}] & \longrightarrow & \mathcal{C}^{k-1, \alpha}(\partial M) \\
	\phi \overline{g} & \longmapsto & \partial_{{\eta}_{\overline{g}}} \phi.
	\end{array}
	$$
	So, $\mathcal{N}_{\overline{g}}^{-1}\left( \left\{ 0 \right\} \right) = [\overline{g}]^{0}$ and the differential $(\mbox{d}\mathcal{N}_{\overline{g}})_{\phi \overline{g}} : \mathcal{C}^{k, \alpha}(M) \longrightarrow \mathcal{C}^{k-1, \alpha}(\partial M)$ is given by $(\mbox{d}\mathcal{N}_{\overline{g}})_{\phi \overline{g}} (\psi) = \partial_{\eta_{\overline{g}}} \psi$, for all $\phi \overline{g} \in [\overline{g}]$ and $\psi \in \mathcal{C}^{k, \alpha}(M)$. Now, by the last Proposition, $\left( \mbox{d} \mathcal{N}_{\overline{g}} \right)_{\phi \overline{g}}$ admits a bounded right-inverse, for all $\phi\overline{g}\in [\overline{g}]$. Therefore, the differential is surjective and its kernel,
	$$\mbox{ker } (\mbox{d}\mathcal{N}_{\overline{g}})_{\phi \overline{g}} = \left\{ \psi \in \mathcal{C}^{k, \alpha}(M) : \partial_{\eta_{\overline{g}}} \psi = 0 \right\},$$
	has a closed complemented in $\mathcal{C}^{k, \alpha}(M)$. It follows that $[\overline{g}]^{0}$ is an embedded submanifold of $[\overline{g}]$.
\end{proof}

We can also combine both features of interest in the same conformal class, defining the $\mathcal{C}^{k, \alpha}$-normalized conformal class consisting of metrics of volume one
$$[\overline{g}]^{0}_{1} = \left\{ \phi \overline{g} : \phi \in \mathcal{C}_{+}^{k, \alpha}(M), \ \partial_{\eta_{\overline{g}}} \phi = 0, \ \displaystyle\int_{M} \phi^{\frac{m}{2}} \ \omega_{\overline{g}} = 1 \right\}.$$
This is an embedded submanifold of $\mathcal{M}^{k, \alpha}(M)_{1}$ and an embedded submanifold of $[\overline{g}]$. We can express $[\overline{g}]^{0}_{1}$ as $[\overline{g}]^{0} \cap \mathcal{M}^{k, \alpha}(M)_{1}$, for instance. The correspondent tangent space is identified with
$$T_{\overline{g}}[\overline{g}]^{0}_{1} = \left\{ \psi \in \mathcal{C}^{k, \alpha}(M)^{0} : \displaystyle\int_{M} \psi \ \omega_{\overline{g}} = 0 \right\}.$$

\subsection{The Hilbert-Einstein Functional}

Consider the Hilbert-Einstein functional $F: \mathcal{M}^{k, \alpha}(M)_{1} \longrightarrow \mathbb{R}$ given by
\begin{equation}
F(\overline{g}) = \displaystyle\int_{M}{R_{\overline{g}} \ \omega_{\overline{g}}}.
\label{functional}
\end{equation}
It is well-known that $F$ is a smooth functional over $\mathcal{M}^{k, \alpha}(M)$ and over $[\overline{g}]$. Let $\overline{g}(t)$ be a variation of a metric $\overline{g} \in \mathcal{M}^{k, \alpha}(M)$ in the direction of $h \in T_{\overline{g}} \mathcal{M}^{k, \alpha}(M)$, that is, $\overline{g}(0) = \overline{g}$ and $\frac{d}{dt}\overline{g}(t)|_{t=0} = h$. Then we can calculate the first variation of the functional $F$,
$$\delta F(\overline{g})h \colon = \left. \frac{d}{dt} \right|_{t=0} F(\overline{g}(t)) = \displaystyle\int_{M} \delta R_{\overline{g}(t)} \ \omega_{\overline{g}} + R_{\overline{g}} \ \delta  \omega_{\overline{g}(t)},$$
where the first variation of scalar curvature is by
$$\delta R_{\overline{g}(t)} = \delta \left( \overline{g}(t)^{ij} R_{ij}(t) \right) = \delta \overline{g}(t)^{ij} R_{ij}(t) + \overline{g}(t)^{ij} \delta R_{ij}(t) = -h^{ij} R_{ij}(t) + g^{ij} \delta R_{ij}(t)$$
whith $R_{ij}(t)$ denoting the coordinates of the Ricci tensor of the metric $\overline{g}(t)$. We have
$$ g^{ij} \delta R_{ij}(t) = \overline{\nabla}_{i} \left( \overline{\nabla}_{j} h^{ij} - \overline{\nabla}^i (\overline{g}^{lm} h_{lm}) \right),$$
and
$$\delta \omega_{\overline{g}(t)} = \displaystyle\frac{1}{2} \overline{g}^{ij}h_{ij} \omega_{\overline{g}},$$
then
$$\delta F(\overline{g})h =  \displaystyle\int_{M} \left( - h_{ij}R_{ij} + \displaystyle\frac{1}{2} \overline{g}^{ij}h_{ij} R_{\overline{g}} \right) \omega_{\overline{g}} + \displaystyle\int_{M} \overline{\nabla}_{i} \left( \overline{\nabla}_{j} h^{ij} - \overline{\nabla}^i (\overline{g}^{lm} h_{lm}) \right) \omega_{\overline{g}},$$
by Divergence Theorem,
$$ \delta F(\overline{g})h = - \displaystyle\int_{M} \left( h_{ij}R_{ij} - \displaystyle\frac{1}{2} \overline{g}^{ij}h_{ij} R_{\overline{g}} \right) \omega_{\overline{g}} - \displaystyle\int_{\partial M} \left( \overline{\nabla}_{j} h^{ij} - \overline{\nabla}^i (\overline{g}^{lm} h_{lm}) \right) \eta_{i} \ \omega_{\overline{g}},$$
where $\eta_{i}$ denote the coordinates of the inward unitary normal vector field, with respect to $\overline{g}$, on the boundary. This expression can be written in a concise form as
$$\delta F(\overline{g})h = - \displaystyle\int_{M} \left\langle Ric_{\overline{g}} - \displaystyle\frac{1}{2}R_{\overline{g}} \overline{g}, h \right\rangle_{\overline{g}} \omega_{\overline{g}} - \displaystyle\int_{\partial M} \left( \overline{\nabla}_{j} h^{ij} \right) \eta_{i} \sigma_{\overline{g}} + \displaystyle\int_{\partial M} \left\langle \overline{\nabla} \mbox{tr}_{\overline{g}} h, \eta \right\rangle_{\overline{g}} \sigma_{\overline{g}}.$$

After some more calculations, we get
$$\delta F(\overline{g})h = - \displaystyle\int_{M} \left\langle Ric_{\overline{g}} - \displaystyle\frac{1}{2}R_{\overline{g}} {\overline{g}},h\right\rangle_{\overline{g}} \omega_{\overline{g}} - 2 \displaystyle\int_{\partial M} \left( \delta H_{\overline{g}} + \frac{1}{2} \left\langle I\!I_{\overline{g}},h \right\rangle\right) \sigma_{\overline{g}}$$

For compact Riemannian manifolds (without boundary) it is well-known that the critical points of $F$ on $\mathcal{M}^k(M)_{1}$ are the Einstein metrics of volume $1$ on $M$ and if $F$ is restricted to the $\mathcal{C}^{k, \alpha}$-conformal classes of volume one $[\overline{g}]_{1}$, then the critical points are those metrics conformal to $\overline{g}$, having unitary volume and constant scalar curvature. We are interested in critical points for the restriction to the $\mathcal{C}^{k, \alpha}$-normalized conformal classes, once we are dealing with manifolds with boundary.

The critical points of $F$ on $[\overline{g}]^{0}_{1}$ are those metrics conformal to $\overline{g}$, having unit volume, constant scalar curvature and vanishing mean curvature. Indeed, when we take $h = \psi \overline{g} \in T_{\overline{g}}[\overline{g}]^{0}_{1}$, the first variation becomes
$$
\begin{array}{rcl}
\delta F(\overline{g}) \psi \overline{g} & = & - \displaystyle\int_{M} \left\langle Ric_{\overline{g}} - \displaystyle\frac{1}{2}R_{\overline{g}} \overline{g}, \psi \overline{g} \right\rangle_{\overline{g}} \ \omega_{\overline{g}} \\
& = & - \displaystyle\int_{M} \left( \psi R_{\overline{g}} - \displaystyle\frac{1}{2} m \psi R_{\overline{g}} \right) \ \omega_{\overline{g}} \\
& = &  \displaystyle\frac{m-2}{2} \displaystyle\int_{M} \psi R_{\overline{g}} \ \omega_{\overline{g}}.
\end{array}
$$
So, by the Fundamental Lemma of Calculus of Variations, $\delta F(\overline{g}) \psi \overline{g} = 0$ for all $\psi \in \mathcal{C}^{k, \alpha}(M)^0$ with null integral if and only if $R_{\overline{g}}$ is constant.

If $\overline{g} \in \mathcal{M}^{k, \alpha}(M)_{1}$ is a critical point of $F$ on $[\overline{g}]^{0}_{1}$, then (see \cite{KO79}) the second variation of $F$ is given by the quadratic form
$$
\begin{array}{rcl}
\delta^2 {F}({\overline{g}})(\psi) &  = & \displaystyle\frac{(m-2)}{2} \displaystyle\int_{M} \left( (m-1) \Delta_{\overline{g}} \psi - R_{\overline{g}} \psi \right) \psi \ \omega_{\overline{g}},
\end{array}
$$
where $\psi \in \mathcal{C}^{k, \alpha}(M)^0$ and has vanishing integral and $\overline{g}$ is nondegenerate if $R_{\overline{g}} = 0$ or $\frac{R_{\overline{g}}}{m-1}$ is not an eigenvalue of $\Delta_{\overline{g}}$ with Neumann boundary conditions.

In fact, note that $(m-1) \Delta_{g} - R_{g}$, as an operator from $\mathcal{C}^{k,\alpha}(M)$ to $\mathcal{C}^{k-2, \alpha}$, is Fredholm of index zero. Note that this operator carries the subspace
$$\left\{ \psi \in \mathcal{C}^{k, \alpha}(M) : \partial_{\eta_{\overline{g}}} \psi = 0, \displaystyle\int_{M} \psi = 0 \right\}$$
into the subspace of $\mathcal{C}^{k-2, \alpha}$ consisting of functions with vanishing integral. It follows that
$$(m-1) \Delta_{\overline{g}} - R_{\overline{g}} :  T_{\overline{g}}[\overline{g}]_{1} \longrightarrow \mathcal{C}^{k-2, \alpha}(M)$$
is Fredholm of index zero. So, the quadratic form $\delta^2 F (\overline{g})(\psi \overline{g},\psi \overline{g})$ is nondegenerate if and only if $\mbox{ker} \left( (m-1) \Delta_{g} - R_{g} \right) = \{0\}$; indeed, the kernel is non-trivial if and only if $\frac{R_{\overline{g}}}{m-1}$ is a non-zero eigenvalue of $\Delta_{\overline{g}}$.

\section{\bf Local Rigidity and Bifurcation of Solutions of the Yamabe Problem}
\def\theequation{3.\arabic{equation}}\makeatother
\setcounter{equation}{0}

\subsection{Local Rigidity}

In this section we present a result obtained in \cite{LI12}, written in a slightly different way to better suit our context. We refer to \cite[Proposition~3 and Corollary~4]{LI12} for details. The proofs are essentially the same. We begin with the following definition.

\begin{definition}
	{\rm Let $\overline{g} \in \mathcal{M}^{k, \alpha}(M)_1$ with constant scalar curvature $R_{\overline{g}}$ in $M$. We say that $\overline{g}$ is \textbf{nondegenerate} if either $R_{\overline{g}} = 0$ or if $\frac{R_{\overline{g}}}{(m-1)}$ is not an eigenvalue of $\Delta_{\overline{g}}$, with the Neumann boundary condition $\partial_{{\eta}_{\overline{g}}} f = 0$. In other words, $\frac{R_{\overline{g}}}{(m-1)}$ is not a solution of the eigenvalue problem
	\begin{equation}
	\left\{
	\begin{array}{rcl}
	\Delta_{\overline{g}} f & = & \lambda f, \mbox{ on $M$} \\
	\partial_{{\eta}_{\overline{g}}} f & = & 0, \mbox{ on $\partial M$}.
	\end{array}
	\right.
	\label{EP}
	\end{equation}}
\end{definition}

\begin{proposition}\label{thm:rigidity}
	Let $\overline{g}_{*} \in \mathcal{M}^{k, \alpha}(M)_{1}$ be a nondegenerate constant scalar curvature metric. Then, there exists an open neighborhood $U$ of $\overline{g}_{*}$ in $\mathcal{M}^{k, \alpha}(M)_{1}$ such that the set
	$$S = \left\{ \overline{g} \in U : R_{\overline{g}} \mbox{ is constant } \right\},$$
	is a smooth embedded submanifold of $\mathcal{M}^{k, \alpha}(M)_{1}$ which is strongly transverse to the $\mathcal{C}^{k, \alpha}$-normalized conformal classes.
\end{proposition}
\noindent
\begin{proof}
\noindent
	The proof is a direct application of \cite[Proposition~1]{LI12}.
\end{proof}

\begin{corollary}
	Let $\overline{g}_{*} \in \mathcal{M}^{k, \alpha}(M)_{1}$ be a nondegenerate metric on $M$  with constant scalar curvature and vanishing mean curvature. Then, there is an open neighborhood $U$ of $\overline{g}_{*}$ in $\mathcal{M}^{k, \alpha}(M)_{1}$ such that every $\mathcal{C}^{k, \alpha}$-normalized conformal class of metrics in $\mathcal{M}^{k, \alpha}(M)_{1}$ has at most one metric of constant scalar curvature and volume one in $U$.
\end{corollary}
\noindent
\begin{proof}
\noindent
	The fact that the manifold $S$ is transverse to the normalized conformal class guarantees the local uniqueness of intersections.
\end{proof}

\subsection{Bifurcation of solutions}

Let $M$ be a $m$-dimensional compact Riemannian manifold with boundary, $m \geq 3$. Define
$$
\begin{array}{rcl}
\left[ a, b \right] & \longrightarrow & \mathcal{M}^{k, \alpha}(M)_{1}, \ \ \ k \geq 3\\
s & \longmapsto & \overline{g}_{s}
\end{array}
$$
a continuous path of Riemannian metrics on $M$ having constant scalar curvature $R_{\overline{g}_{s}}$ and vanishing mean curvature $H_{\overline{g}_{s}}$, for all $s \in [a,b]$.

\begin{definition}
	{\rm An instant $s_{*} \in [a, b]$ is called a \textbf{bifurcation instant} for the family $\{\overline{g}_{s}\}_{s \in [a,b]}$ if there exists a sequence $(s_{n})_{n \geq 1} \subset [a,b]$ and a sequence $(\overline{g}_{n})_{n \geq 1} \subset \mathcal{M}^{k, \alpha}(M)_{1}$ of Riemannian metrics on $M$ satisfying:
	\begin{enumerate}
		\item[(a)] $\displaystyle\lim_{n \rightarrow \infty} s_{n} = s_{*}$ and $\displaystyle\lim_{n \rightarrow \infty} \overline{g}_{n} = \overline{g}_{s_{*}} \in \mathcal{M}^{k, \alpha}(M)_{1}$;
		\item[(b)] $\overline{g}_{n} \in [\overline{g}_{s_{n}}]$, but $\overline{g}_{n} \neq \overline{g}_{s_{n}}$, for all $n \geq 1$;
		\item[(c)] $\overline{g}_{n}$ has constant scalar curvature and vanishing mean curvature, for all $n \geq 1$.
	\end{enumerate}
	If $s_{*} \in [a,b]$ is not a bifurcation instant, the family $\{\overline{g}_{s}\}_{s \in [a,b]}$ is said \textbf{locally rigid} at $s_{*}$.}
\end{definition}

An instant $s \in [a,b]$ for which $\frac{R_{\overline{g}_{s}}}{(m-1)}$ is a non-vanishing solution of problem (\ref{EP}) is called a \textbf{degeneracy instant} for the family $\{g_{s}\}_{s \in [a,b]}$.

\begin{theorem}\label{thm:bif}
	Let $M$ be a $m$-dimensional compact manifold, with boundary $\partial M \neq \emptyset$, $m \geq 3$, and let
	$$
	\begin{array}{rcl}
	\left[ a, b \right] & \longrightarrow & \mathcal{M}^{k, \alpha}(M)_{1}, \ \ \ k \geq 3 \\
	s & \longmapsto & \overline{g}_{s}
	\end{array}
	$$
	be a $C^1$-path of Riemannian metrics on $M$ having constant scalar curvature $R_{\overline{g}_{s}}$ and vanishing mean curvature $H_{\overline{g}_{s}}$. Denote by $n_{s}$ the number of eigenvalues of the Laplace-Beltrami operator $\Delta_{\overline{g}_{s}}$, with Neumann boundary condition (counted with multiplicity), that are less than $\frac{R_{\overline{g}_{s}}}{(m-1)}$. Assume that, if $s = a$ or $s=b$, $\frac{R_{\overline{g}_{s}}}{(m-1)} = 0$ or it is not an eigenvalue of $\Delta_{\overline{g}_{s}}$ and $n_{a} \neq n_{b}$. Then, there exists a bifurcation instant $s_{*} \in (a,b)$ for the family $\{\overline{g}_{s}\}_{s \in [a,b]}$.
\end{theorem}
\noindent
\begin{proof}
\noindent
	The result follows from the non-equivariant bifurcation theorem \cite[Theorem~A.2, Appendix A]{PI12}.
\end{proof}

\begin{remark}
{\rm Given a Riemannian manifold $(M, \overline{g})$ with minimal boundary $\partial M \neq \emptyset$, it is important to stress that for all $s \in \mathbb{R}^+$, the manifold $(M, s \overline{g})$ also has minimal boundary, once $H_{s {\overline{g}}} = \frac{1}{\sqrt{s}} H_{{\overline{g}}}$. Moreover, $\Delta_{s \overline{g}} = \frac{1}{s} \Delta_{\overline{g}}$, $R_{s \overline{g}} = \frac{1}{s} R_{\overline{g}}$ and $\eta_{s \overline{g}} = \frac{1}{\sqrt{s}} \eta$. This means that the spectrum of the operator  $\Delta_{\overline{g}} - \frac{R_{\overline{g}}}{m-1}$, with Neumann boundary condition, is invariant by homothety of the metric. On the other hand, $\omega_{s \overline{g}} = s^{\frac{m}{2}} \omega_{\overline{g}}$. When needed we will normalize metrics to have volume $1$ without change the spectral theory of the operator $\Delta_{\overline{g}} - \frac{R_{\overline{g}}}{m-1}$, with Neumann boundary condition.}
\end{remark}

\section{\bf Bifurcation of Solutions of the Yamabe Problem in Product Manifolds}
\def\theequation{4.\arabic{equation}}\makeatother
\setcounter{equation}{0}

Let $(M_{1}, g^{(1)})$ be a compact Riemannian manifold, with $\partial M_{1} = \emptyset$ and constant scalar curvature, and let $(M_{2}, \overline{g}^{(2)})$ be a compact Riemannian manifold with minimal boundary and constant scalar curvature. Consider the product manifold, $M = M_{1} \times M_{2}$, which boundary is given by $\partial M = M_{1} \times \partial M_{2}$. Let $m_{1}$ and $m_{2}$ be the dimensions of $M_{1}$ and $M_{2}$, respectively, and assume that $\mbox{dim}(M) = m = m_{1} + m_{2} \geq 3$. For each $s \in (0, + \infty)$, define ${\overline{g}}_{s} = {g}^{(1)} \oplus s {\overline{g}}^{(2)}$ a family of metrics on $M$. Then $\left\{ {\overline{g}}_{s} \right\}_{s} \subset \mathcal{M}^{k, \alpha}(M)$.

The following statements, briefly justified, are valid.
\begin{itemize}
	\item[(a)] $(M, \overline{g}_{s})$ has constant scalar curvature, for all $s>0$, and its scalar curvature is given by
	$$R_{{\overline{g}}_{s}} = R_{{g}^{(1)}} + R_{s {\overline{g}}^{(2)}} = R_{{g}^{(1)}} + \frac{1}{s} R_{{\overline{g}}^{(2)}}.$$
	\item[(b)] Since we can identify the tangent space of the product manifold $T_{(p,q)}M$ with the direct sum $T_{p} M_{1} \oplus T_{q} M_{2}$, for $p \in M_{1}$ and $q \in M_{2}$, the interior vector field $\eta_{s}$, normal to $\partial M$, can be written as
	$$\eta_{s} = 0 + \frac{1}{\sqrt{s}} \eta_{2},$$
	where $\eta_{2} $ is the interior vector field normal to $\partial M_{2}$.
	\item[(c)] The mean curvature of $\partial M$ is zero, once we have
	$$H_{{\overline{g}}_{s}} = H_{s {\overline{g}}^{(2)}} = \frac{1}{\sqrt{s}} H_{{\overline{g}}^{(2)}}.$$
	\item[(d)] Laplace-Beltrami operator, with respect to ${\overline{g}}_{s}$, is given by
	$$\Delta_{{\overline{g}}_{s}} = \left( \Delta_{{g}^{(1)}} \otimes I \right) + \displaystyle\frac{1}{s} \left( I \otimes \Delta_{{\overline{g}}^{(2)}} \right).$$
	\item[(e)] Consider the family $\{{\overline{g}}_{s}\}_{s >0}$ for the purpose of study bifurcation instants causes no loss of generality. In fact, $s_{*}$ is a bifurcation instant for the family $\left\{ {g}^{(1)} \oplus s {\overline{g}}^{(2)} \right\}_{s>0}$ if and only if $s_{*}$ is a bifurcation instant for the family $\left\{ \frac{1}{s} {g}^{(1)} \oplus {\overline{g}}^{(2)} \right\}_{s>0}$ on $M$. The same is valid for degeneracy instants.
\end{itemize}

Now, considering the remark at the end of previous section, $\{\overline{g}_{s}\}_{s>0}$ is a family of critical points of the functional
$$F: \mathcal{M}^{k, \alpha}(M)_{1} \longrightarrow \mathbb{R},$$
restricted to $[\overline{g}]^{0}_{1}$, so it is a family of solutions with minimal boundary of the Yamabe problem in product manifolds.

In order to investigate the existence of bifurcation instants for the family $\{\overline{g}\}_{s > 0}$, we are interested in study the spectrum of the operator $\mathcal{J}_{s}$
$$\mathcal{J}_{s} = \Delta_{{\overline{g}}_{s}} - \displaystyle\frac{R_{{\overline{g}}_{s}}}{m-1},$$
whose domain is $\left\{ \psi \in \mathcal{C}^{k, \alpha}(M)^0 : \displaystyle\int_{M} \psi \ \omega_{\overline{g}_s} = 0 \right\}$. Denote by $0 = \rho_{0}^{(1)} < \rho_{1}^{(1)} < \rho_{2}^{(1)} < \ldots $ the sequence of all distinct eigenvalues of $\Delta_{{g}^{(1)}}$, with geometric multiplicity $\mu_{i}^{(1)}$, $i \geq 0$, and by $0 = \rho_{0}^{(2)} < \rho_{1}^{(2)} < \rho_{2}^{(2)} < \ldots $ the sequence of all distinct eigenvalues of $\Delta_{{\overline{g}}^{(2)}}$, subjected to Neumann boundary condition,
\begin{equation}
\left\{
\begin{array}{rcl}
\Delta_{{\overline{g}}^{(2)}} f^{(2)} & = & \rho_{j}^{(2)} f^{(2)},\mbox{ on $ M$,} \\
\partial_{\eta_2} f^{(2)} & = & 0, \mbox{ on $\partial M$},
\end{array}
\right.
\label{EP2}
\end{equation}
where $j \geq 0$, with $\mu_{j}^{(2)}$ the geometric multiplicity of $\rho_{j}^{(2)}$, $j  \geq 0$. Then, the spectrum of $\mathcal{J}_{s}$ is given by
$$\Sigma(\mathcal{J}_{s}) = \left\{ \sigma_{i,j} : i, j \geq 0, i + j > 0 \right\},$$
where
$$\sigma_{i,j}(s) = \rho_{i}^{(1)} + \displaystyle\frac{1}{s} \rho_{j}^{(2)} - \displaystyle\frac{1}{m-1} \left( R_{{g}^{(1)}} + \displaystyle\frac{1}{s} R_{{\overline{g}}^{(2)}} \right), $$
are the eigenvalues of $J_{s}$, with Neumann boundary condition on $\partial M$, with geometric multiplicity equal to the product $\mu_{i}^{(1)} \mu_{j}^{(2)}$.

We emphasize that $\sigma_{i,j}$'s are not necessarily all distinct!

\begin{definition}
	{\rm Let $i_{*}$ and $j_{*}$ be the smallest non-negative integers that satisfy
	$$
	\rho_{i_{*}}^{(1)} \geq \displaystyle\frac{R_{{g}^{(1)}}}{m-1}, \ \ \ \rho_{j_{*}}^{(2)} \geq \displaystyle\frac{R_{{\overline{g}}^{(2)}}}{m-1}.
	$$
	We say that the pair of metrics $({g}^{(1)}, {\overline{g}}^{(2)})$ is {\bf degenerate} if equalities hold in both cases, that is, $\sigma_{i_{*}, j_{*}}(s) = 0$, for all $s$. In this situation, the operator $\mathcal{J}_{s}$ is also called degenerate.}
\end{definition}

We can state that, if $R_{{g}^(1)} < 0$ or if $R_{{\overline{g}}^{(2)}} < 0$ and $H_{{\overline{g}}^(2)} = 0$, then $({g}^{(1)}, {\overline{g}}^{(2)})$ is, certainly, non-degenerate. Observe also that, if $({g}^{(1)}, {\overline{g}}^{(2)})$ is degenerate, zero is an eigenvalue of $\mathcal{J}_{s}$, for all $s \in (0, +\infty)$, otherwise, there is only a discrete countable set, $S \subset (0, +\infty)$, of instants $s$ for which the operator $\mathcal{J}_{s}$ is singular. First, we consider the case that both scalar curvatures are positive.

\subsection{The case of positive scalar curvature}

We are interested in studying the zeros of the function $s \mapsto \sigma_{i,j}(s)$, as $i, j$ vary. At first glance, we can already draw some conclusions; for instance, if the function $\sigma_{i,j}$ is not identically zero, for fixed $i,j$, it has at most one zero in $(0, + \infty)$. Let us write $\sigma_{i,j}$ as
$$\sigma_{i,j}(s) = \rho_{i}^{(1)} - \displaystyle\frac{R_{{g}^{(1)}}}{m-1}  + \displaystyle\frac{1}{s} \left(  \rho_{j}^{(2)} - \displaystyle\frac{R_{{\overline{g}}^{(2)}}}{m-1}  \right). $$
Derive
$$\sigma'_{i,j}(s) = - \displaystyle\frac{1}{s^2} \left(  \rho_{j}^{(2)} - \displaystyle\frac{ R_{{\overline{g}}^{(2)}} }{m-1}\right), $$
so, $\sigma'_{i,j}(s) = 0$ if and only if $\rho_{j}^{(2)} = \displaystyle\frac{R_{{\overline{g}}^{(2)}}}{m-1}$ if and only if $R_{{\overline{g}}^{(2)}}$ is a solution of (\ref{EP2}); and if $\sigma'_{i,j}(s) = 0$ for some $s$, then $\sigma'_{i,j}(s) = 0$ for all $s \in (0, +\infty)$. Hence, $\sigma$ is strictly monotone or constant, so if it is not identically null (constant), it certainly has at most one zero (strictly monotone).

\begin{lemma}\label{Lemma}
	Assume that the pair $({g}^{(1)}, {\overline{g}}^{(2)})$ is non-degenerate and that \linebreak $R_{{g}^{(1)}}, R_{{\overline{g}}^{(2)}} > 0$, with $H_{{\overline{g}}^{(2)}}=0$. Then the functions $\sigma_{i,j}(s)$ satisfy the following properties.
	\begin{enumerate}
		\item[(a.)] For all $i, j \geq 0$, the map $s \mapsto \sigma_{i,j}(s)$ is strictly monotone in $(0, +\infty)$, except the maps $\sigma_{i,j_{*}}$, that are constants equal to $\rho_{i}^{(1)} - \displaystyle\frac{R_{{g}^{(1)}}}{m-1}$, when $\rho_{j_{*}}^{(2)} = \displaystyle\frac{R_{{\overline{g}}^{(2)}}}{m-1}$.
		\item[(b.)] For $i \neq i_{*}$ and $j \neq j_{*}$, the map $\sigma_{i,j}(s)$ admits a zero if and only if:
		\begin{itemize}
			\item either $j < j_{*}$ and $i > i_{*}$, in which case $\sigma_{i,j}$ is strictly increasing,
			\item or if $j>j_{*}$ and $i < i_{*}$, in which case $\sigma_{i,j}$ is strictly decreasing.
		\end{itemize}
		\item[(c.)] If $\rho_{i_{*}}^{(1)} = \displaystyle\frac{R_{{g}^{(1)}}}{m-1}$, then $\sigma_{i_{*}, j}$ does not have zeros for any $j \in (0, +\infty)$. If $\rho_{i_{*}}^{(1)} > \displaystyle\frac{R_{{g}^{(1)}}}{m-1}$, then $\sigma_{i_{*},j}$ has a zero if and only if $j < j_{*}$.
		\item[(d.)] If $\rho_{j_{*}}^{(2)} = \displaystyle\frac{R_{\overline{g}}^{(2)}}{m-1}$, then $\sigma_{i, j_{*}}$ does not have zeros for any $i \in (0, +\infty)$. If $\rho_{j_{*}}^{(2)} > \displaystyle\frac{R_{{\overline{g}}^{(2)}}}{m-1}$, then $\sigma_{i,j_{*}}$ has a zero if and only if $i < i_{*}$.
	\end{enumerate}
\end{lemma}
\noindent
\begin{proof}
\noindent
	The entire statement follows directly from a straightforward analysis of the expression
	$$\sigma_{i,j}(s) = \left( \rho_{i}^{(1)} - \displaystyle\frac{R_{{g}^{(1)}}}{m-1} \right) + \displaystyle\frac{1}{s} \left( \rho_{j}^{(2)} - \displaystyle\frac{R_{{\overline{g}}^{(2)}}}{m-1} \right).$$
\end{proof}

\begin{corollary}\label{thm:cor}
	If $({g}^{(1)}, {\overline{g}}^{(2)})$ is nondegenerate, then the set $S \subset (0, +\infty)$, of instants $s$ at which $\mathcal{J}_{s}$ is singular, is countable and discrete. More precisely, it consists of a strictly decreasing sequence $(s_{n}^{(1)})_{n}$ tending to $0$ and a strictly increasing unbounded sequence $(s_{n}^{(2)})_{n}$. For all other values of $s$, $\mathcal{J}_{s}$ is an isomorphism and in particular, the family $\{{\overline{g}}_{s}\}_{s>0}$ is locally rigid at these instants.
\end{corollary}
\noindent
\begin{proof}
\noindent
	By Lemma \ref{Lemma}, each function $\sigma_{i,j}$ has at most one zero, thus there is only a countable number of degeneracy instants for $\mathcal{J}_{s}$. Let $s_{ij}$ be the zero of $\sigma_{ij}$, then
	$$0 < s_{ij} = - \displaystyle\frac{\rho^{(2)}_{j} - \frac{R_{{\overline{g}}^{(2)}}}{m-1}}{\rho_{i}^{(1)} - \frac{R_{{g}^{(1)}}}{m-1}}. $$
	Now, we study the behavior of these zeros in two cases:
	\begin{itemize}
		\item if $j>j_{*}$ and $i<i_{*}$, as $j \rightarrow +\infty$,
		$$s_{ij} = \displaystyle\frac{\rho^{(2)}_{j} - \frac{R_{{\overline{g}}^{(2)}}}{m-1}}{\frac{R_{{g}^{(1)}}}{m-1} - \rho_{i}^{(1)}} \geq \left( \rho^{(2)}_{j} - \displaystyle\frac{R_{{\overline{g}}^{(2)}}}{m-1} \right) \displaystyle\frac{1}{\frac{R_{{g}^{(1)}}}{m-1} - \rho_{i_{*} - 1}^{(1)}} \rightarrow +\infty. $$
		\item if $i>i_{*}$ and $j<j_{*}$, as $i \rightarrow +\infty$,
		$$s_{ij} = \displaystyle\frac{\frac{R_{{\overline{g}}^{(2)}}}{m-1} - \rho^{(2)}_{j} }{\rho_{i}^{(1)} - \frac{R_{{g}^{(1)}}}{m-1}} \leq \displaystyle\frac{1}{\rho_{i}^{(1)} - \frac{R_{{g}^{(1)}}}{m-1}} \left( \displaystyle\frac{R_{{\overline{g}}^{(2)}}}{m-1} - \rho^{(2)}_{1} \right)  \rightarrow 0. $$
	\end{itemize}
	Therefore, all the zeros of the eigenvalues $\sigma_{i,j}$ accumulates only at $0$ and at $+\infty$. Let $s_{*} \in (0, +\infty) \backslash S$. Then, $\mathcal{J}_{s_{*}}$ is an isomorphism, that is, $0 \notin \Sigma(J_{s*})$ or, equivalently, $\displaystyle\frac{R_{\overline{g}_{s_{*}}}}{m-1}$ is not an eigenvalue of $\Delta_{\overline{g}_{s_{*}}}$. So, $\overline{g}_{s_{*}} \in \{{\overline{g}}_{s}\}_{s>0}$ is a nondegenerate metric. It follows from Proposition~\ref{thm:rigidity} that the family $\{{\overline{g}}_{s}\}_{s>0}$ is locally rigid at $s_{*}$.
\end{proof}

Note that $R_{\overline{g}}$ is obviously different from zero, once we are considering only \textsl{positive} scalar curvature.

\begin{theorem}\label{thm:principal}
	Let $(M_{1}, {g}^{(1)})$ be a compact Riemannian manifold with positive constant scalar curvature and $(M_{2}, {\overline{g}}^{(2)})$ a compact manifold with boundary, having positive constant scalar curvature and minimal boundary $\partial M_2$. Assume that the pair $({g}^{(1)}, {\overline{g}}^{(2)})$ is nondegenerate. For all $s \in (0, +\infty)$, let $\overline{g}_{s} = {g}^{(1)} \oplus s {\overline{g}}^{(2)}$ be the metric on the product manifold with boundary, $M = M_{1} \times M_{2}$. Then there exist a sequence tending to $0$ and  a sequence tending to $+\infty$ consisting of bifurcation instants for the family $\{\overline{g}_{s}\}_{s>0}$.
\end{theorem}
\noindent
\begin{proof}
\noindent
	We prove that bifurcation instants consist of subsequences of the two sequences of instants where $\mathcal{J}_{s}$ is singular, whose existence was proved above.
	
	With the same notation used in Corollary~\ref{thm:cor}, let $n_{0} > 0$ be such that $s_{n}^{(1)} < s_{1}^{(2)}$ and $s_{1}^{(1)} < s_{n}^{(2)}$, for all $n > n_{0}$. Then, there is a $\varepsilon > 0$, for all $n > n_{0}$, such that the operator $\mathcal{J}_{( \cdot )}$ is an isomorphism on the intervals $[s_{n}^{(1)} - \varepsilon, s_{n}^{(1)} + \varepsilon]$ and $[s_{n}^{(2)} - \varepsilon, s_{n}^{(2)} + \varepsilon]$, except for the instants $s_{n}^{(1)}$ and $s_{n}^{(2)}$ themselves.
	
	As the zeros of the increasing eigenvalue functions accumulate at $0$ and the zeros of the decreasing eigenvalue functions accumulate at $\infty$, if $\sigma_{p,q}$ is a non-increasing eigenvalue function, for all $s \in (0, s_{n}^{(1)} + \varepsilon]$, $n > n_{0}$, we have $\sigma_{p,q}(s) \neq 0$. So, $\sigma_{p,q}(s_{n}^{(1)} - \varepsilon) < 0$ if and only if $\sigma_{p,q}(s_{n}^{(1)} + \varepsilon) < 0$. On other hand, if we consider an increasing eigenvalue function $\sigma_{i,j}$, for all $n > n_{0}$, we have $\sigma_{i,j}(s_{n}^{(1)}) = 0$, $\sigma_{i,j}(s_{n}^{(1)} - \varepsilon) < 0$, $\sigma_{i,j}(s_{n}^{(1)} + \varepsilon) > 0$ and the fact that $s_{n}^{(1)} < s_{1}^{(2)}$ ensures that there is no decreasing function that vanishes at $s_{n}^{(1)}$. Hence, we can surely conclude that $n_{s_{n}^{(1)} - \varepsilon} \neq n_{s_{n}^{(1)} + \varepsilon}$. By Theorem~\ref{thm:bif} it follows that the subsequence $(s_{n}^{(1)})_{n > n_{0}}$ is the sought sequence of bifurcation instants tending to $0$.
	
	Now, analyzing non-decreasing eigenvalue function and decreasing eigenvalue function in a similar way, we obtain $n_{s_{n}^{(2)} - \varepsilon} \neq n_{s_{n}^{(2)} + \varepsilon}$ and can apply Theorem~\ref{thm:bif} to conclude that the subsequence $(s_{n}^{(2)})_{n > n_{0}}$ is the sought sequence of bifurcation instants tending to $\infty$.
\end{proof}

Note that the case of degenerate pairs cannot be treated with Theorem~\ref{thm:bif}, because in this case, the zero is present in the spectrum of the operator $\mathcal{J}_{s}$, for all $s \in (0, +\infty)$, and thus the hypothesis of the theorem are never satisfied. Another interesting observation is that at degeneracy instants $s$ between $s_{1}^{(1)}$ and $s_{1}^{(2)}$ we do not know, by Theorem ~\ref{thm:principal}, if bifurcation occurs.

\subsection{The case of non-positive scalar curvature}

Now, consider the family $\{\overline{g}_{s}\}_{s>0}$ on the product manifold $M_{1} \times M_{2}$ in the case when one or both scalar curvatures, $R_{g^{(1)}}$ or $R_{\overline{g}^{(2)}}$, are non-positive, maintaining the Neumann boundary condition on $\partial M_{2}$. Observe that, if $R_{g^{(1)}}$ and $R_{\overline{g}^{(2)}}$ are both non-positive, the pair $(g^{(1)}, \overline{g}^{(2)})$ is nondegenerate.

If $R_{g^{(1)}} \leq 0$ and $R_{\overline{g}^{(2)}} > 0$, then the pair $(g^{(1)}, \overline{g}^{(2)})$ is degenerate if and only if $R_{g^{(1)}} = 0$ and $\rho_{j_{*}}^{(2)} = \frac{R_{{\overline{g}}^{(2)}}}{m-1}$, as $\rho_{i}^{(1)} \geq 0$, $\forall i \in \{0,1,2,\ldots\}$.

\begin{theorem}\label{thm:secondmain} The following statements are true.
	\begin{itemize}
		\item[(a)] If $R_{g^{(1)}} \leq 0 $ and $R_{\overline{g}^{(2)}} \leq 0$, then the family $\{\overline{g}_{s}\}_{s>0}$ has no degeneracy instants, and thus it is locally rigid at every $s \in (0, +\infty)$.
		\item[(b)] If $R_{g^{(1)}} \leq 0 $, $R_{\overline{g}^{(2)}} > 0$, and the pair $(g^{(1)}, \overline{g}^{(2)})$ is nondegenerate, then the set of degeneracy instants for $\mathcal{J}_{s}$ is a strictly decreasing sequence $(s_{n})_{n \in \mathbb{N}}$ that converges to $0$ as $n \longrightarrow \infty$. Moreover, every degeneracy instant is a bifurcation instant for the family $\{\overline{g}_{s}\}_{s>0}$
		\item[(c)]Symmetrically, if $R_{g^{(1)}} > 0 $, $R_{\overline{g}^{(2)}} \leq 0$, and the pair $(g^{(1)}, \overline{g}^{(2)})$ is nondegenerate, then the set of degeneracy instants for $\mathcal{J}_{s}$ is a strictly increasing unbounded sequence $(s_{n})_{n \in \mathbb{N}}$ and every degeneracy instant is a bifurcation instant for the family $\{\overline{g}_{s}\}_{s>0}$.
	\end{itemize}
\end{theorem}
\noindent
\begin{proof}
\noindent
	The result follows from an analysis of the functions
	$$\sigma_{i,j}(s) = \rho_{i}^{(1)} + \displaystyle\frac{1}{s} \rho_{j}^{(2)} - \displaystyle\frac{1}{m-1} \left( R_{{g}^{(1)}} + \displaystyle\frac{1}{s} R_{{\overline{g}}^{(2)}} \right). $$
	
	In case (a), it is straightforward to see that $\sigma_{i,j}(s) > 0$ for all $i, j = 0, 1, 2,\ldots,$, with $i+j > 0$, so $\mathcal{J}_{s}$ has no vanishing eigenvalues and the result follows.
	
	In case (b), the functions $\sigma_{i,j}$ admit a zero only if $i \geq 0$ and $j < j_{*}$. If $s_{i,j}$ denotes the zero of such a function, we have
	$$ 0 < s_{i,j} = \left| \frac{\rho_{j}^{(2)} - \frac{R_{\overline{g}^{(2)}}}{m-1}}{\rho_{i}^{(1)} - \frac{R_{g^{(1)}}}{m-1}} \right| \leq \displaystyle\frac{R_{{\overline{g}}^{(2)}}}{\rho_{i}^{(1)} - \frac{R_{{g}^{(1)}}}{m-1} } \longrightarrow 0, \mbox{  as  } i \longrightarrow +\infty$$
	Hence, there is a decreasing sequence $(s_{n})_{n \in \mathbb{N}}$ of degeneracy instants that accumulates on $0$. Since $(s_{n})_{n \in \mathbb{N}}$ accumulates only at zero, for each $n \in \mathbb{N}$, exists $\varepsilon > 0$ such that the interval $[s_{n} - \varepsilon, s_{n} + \varepsilon]$ contains only $s_{n}$ as a degeneracy instant. Arguing as in the proof of Theorem~\ref{thm:principal}, we have $n_{{s_{n}} - \varepsilon} \neq n_{{s_{n}} + \varepsilon}$. The conclusion follows from Theorem~\ref{thm:bif}.
	
	Proceeding in a similar way, in case (c), the functions $\sigma_{i,j}$ admit a zero only if $i < i^*$ and $j \geq 0$, and we have an unbounded increasing sequence of degeneracy instants, each of which is a bifurcation instant for the family $\{\overline{g}_{s}\}_{s>0}$.
\end{proof}

\end{document}